\documentclass[10pt]{amsart}

\usepackage{amsmath, amssymb, amsthm, mathtools}

\DeclareMathOperator{\dist}{\mathsf{dist}}
\newtheorem{thm}{Theorem}[section]
 \newtheorem{proposition}[thm]{Proposition}  
\newtheorem{lemma}[thm]{Lemma}
\newtheorem{lem}[thm]{Lemma}
\newtheorem{corollary}[thm]{Corollary}
\newtheorem{prop}[thm]{Proposition}

\newtheorem{definition}[thm]{Definition}
\newtheorem{remark}[thm]{Remark}

\numberwithin{equation}{section}

\title{A Note on Hyperbolically Embedded Subgroups}
\author{Eduardo Mart\'inez-Pedroza and Farhan Rashid}
\date{\today}
\keywords{Hyperbolically embedded, Relatively hyperbolic, Acylindrically hyperbolic, Fine graph}
\subjclass[2010]{20F65, 20F67}

\begin{document}

\maketitle

\begin{abstract}
Let $G$ be a group and $H$ a subgroup of $G$. The notion of $H$ being hyperbolically embedded in $G$ was introduced by Dahmani, Guiraldel and Osin.  This note introduces an equivalent definition of hyperbolic embedded subgroup based on Bowditch's approach to relatively hyperbolic groups in terms of fine graphs.     
\end{abstract}

\section{Introduction}
Hyperbolically embedded subgroups were introduced by Dahmani, Guiraldel and Osin~\cite{osin2}, their definition is recalled in Section~\ref{sec:HypEmbDef}. Given a group $G$, $X\subset G$ and $H\leq G$, let $H\hookrightarrow_h (G,X)$ denote that $H$ is a   hyperbolically embedded subgroup of $G$ with respect to $X$. The class of groups containing a hyperbolically embedded subgroup includes hyperbolic groups, relatively hyperbolic groups, and some non-relatively hyperbolic examples including most mapping class groups of surfaces. 
 
Let $\Gamma$ be a simplicial graph, let $v$ be a vertex of  $\Gamma$, and let \begin{align*}
T_v \Gamma = \{w \in V(\Gamma) \mid \{v,w\}\in E(\Gamma)\}.
\end{align*}
denote the set of the  vertices adjacent to $v$.
For $x,y \in T_v \Gamma$,  the \emph{angle metric} $\angle_v (x,y)$ is the combinatorial length of the shortest path in the graph  $\Gamma - \{v\}$ between $x$ and $y$, with $\angle_v (x,y) = \infty$ if there is no such path.  The graph $\Gamma$ is \emph{fine at $v$} if $(T_v \Gamma,\angle_v)$ is a locally finite metric space. 
The graph $\Gamma$ is   \emph{fine} if it is fine at every vertex, this notion  was introduced by Bowditch~\cite{bowditch}. 

There is an approach to relatively hyperbolic groups by Bowditch based on the notion of \emph{fine graph}~\cite{bowditch}. The main result of the note is a characterization of hyperbolically embedded subgroups that generalizes Bowditch's cited approach. 

\begin{definition}\label{def:GHgraph}
Let $G$ be a group and let $H$ be a subgroup. 
A graph $\Gamma$ is a $(G,H)$-graph if $G$ acts on $\Gamma$ and
\begin{enumerate}
\item $\Gamma$ is connected and hyperbolic,
\item there are finitely many $G$-orbits of vertices, 
\item $G$-stabilizers of vertices are finite or  conjugates of $H$,
 and there is a vertex with $G$-stabilizer   equals $H$,
\item $G$-stabilizers  of edges are finite, and
\item \label{item:05} $\Gamma$ is fine at
$V_\infty (\Gamma) = \{v \in V(\Gamma) \mid v \text{ has infinite stabilizer} \}$.
\end{enumerate}
\end{definition}

Bowditch's definition~\cite[Definition 2]{bowditch} of relative hyperbolicity can be phrased as follows: A group {$G$ is hyperbolic relative to a subgroup of $H$} if and only if there exists a $(G,H)$-graph that has finitely many $G$-orbits of edges. This note proves the following result.

\begin{thm}\label{thm:main}
Let $G$ be a finitely generated group.
An infinite subgroup $H$ is hyperbolically embedded in $G$  if and only if there exists a $(G,H)$-graph.
\end{thm}

The ``only if" direction of Theorem~\ref{thm:main} is straightforward. In the case that $H\hookrightarrow_h (G,S)$ then the Coned-off Cayley graph $\hat\Gamma(G,H,S)$, defined below,
is a $(G,H)$-graph by Proposition~\ref{q2}  which is discussed in Section~\ref{sec:HypEmbDef}. 

\begin{definition}[Coned-off Cayley Graph]
Let $G$ be a group, $X\subset G$ and $H\leq G$. The \emph{Coned-off Cayley graph $\hat \Gamma(G,H,X)$} is the graph with vertex set the $G$-set $G\cup G/H$, where
$G/H$ is the set of left cosets of $H$ in $G$; and edge set $\{\{g,gx\}\colon g\in G \text{ and }  x\in X\} \cup \{ \{g, gH\} \colon g\in G \}$.  Vertices of $\hat \Gamma(G,H,X)$ in $G/H$ are called \emph{cone vertices}. The notion of coned-off Cayley graph was introduced by Farb~\cite{farb}.
\end{definition}
\begin{proposition} \label{q2}
Let $G$ be a group, $X\subset G$ and $H\leq G$. Then $H \hookrightarrow_h (G,X) $ if and only if   $\hat{\Gamma}(G,H,X)$ is connected, hyperbolic, and fine at cone vertices. 
\end{proposition}

A group $G$ is finitely generated relative to the subgroup $H$ if there is a finite $S\subset G$ such that $G=\langle S\cup H\rangle$. The converse of Theorem~\ref{thm:main} follows from Proposition~\ref{q2} and the main technical result of the note:  

\begin{prop}\label{fin}
Suppose $G$ is finitely generated relative to the infinite subgroup $H$. If there is a $(G,H)$-graph then there is a subset $X$ of $G$ such that the Coned-off Cayley graph $\hat{\Gamma} (G,H,X)$ is a $(G,H)$-graph.
\end{prop}

The skeleton of the proof of Proposition~\ref{fin} is as follows. We define the notion of \emph{thick} $(G,H)$-graph, see Definition~\ref{def:thick}. In Section~\ref{sec:05}, we prove that if there is a $(G,H)$-graph, then there is a thick $(G,H)$-graph, see Proposition~\ref{prop:Existencethick}; and  that if there is a thick $(G,H)$-graph then there is $X\subset G$ such that  $\hat{\Gamma} (G,H,X)$ is a $(G,H)$-graph, see Proposition~\ref{prop:end}. These two results yield Proposition~\ref{fin} and the main result follows.
 
Roughly speaking, a thick $(G,H)$-graph is a $(G,H)$-graph $\Delta$ that contains as a $G$-subgraph the coned-off Cayley graph $\Gamma(G,H,S)$ for $S$ a finite relative generating set of $G$ with respect to $H$. To construct a thick graph from a $(G,H)$-graph we perform a finite number of attachments of $G$-orbits of edges and vertices. Proving that attaching a new $G$-orbit of edges to a $(G,H)$-graph produces a new $(G,H)$-graph is a non-elementary argument which is the content of Section~\ref{sec:ModifyingGH-graphs}. Specifically, Proposition~\ref{lem2.2hs} generalizes results from Wise and the first author~\cite[Lemma 2.9]{MPW}, and Bowditch~\cite[Lemma 2.3]{bowditch}; the arguments in those references do not seem to generalize to our broader context and a new strategy was required.

\subsection*{Acknowledgments.} We thank Hadi Bigdely for proof reading parts of the manuscript. We also thank the referee for comments and corrections. Some of the results of this  note are based on the Master thesis of the second author at Memorial University of Newfoundland under the supervision of the first author~\cite{Rashid20}. In that work,  an attempt to prove a weaker version of Theorem~\ref{thm:main} is outlined, using a more restrictive  notion of $(G,H)$-graph. 
 The first author acknowledges funding by the Natural Sciences and Engineering Research Council of Canada NSERC.

\section{Preliminaries}

Most graphs considered in this note are $1$-dimensional simplicial complexes. The only exception is in Section~\ref{sec:HypEmbDef} where we recall the definition of hyperbolically embedded subgroup from~\cite{osin2} and prove Proposition~\ref{fin} which allow us to consider only simplicial graphs for the rest of the manuscript. 

A \emph{graph} is an ordered pair $(V,E)$, where $V$ is a set, and $E$ is a relation on $V$ that is anti-reflexive and symmetric. Elements of set $V$ are called \textit{vertices}, and elements of set $E$ are called \textit{edges}. For a graph $\Gamma$, we denote $V(\Gamma)$ and $E(\Gamma)$ its vertex and edge set, respectively. If $v\in V(\Gamma)$, $e\in E(\Gamma)$ and $v\in e$, then $v$ is \textit{incident} to $e$. Vertices incident to the same edge are called \textit{adjacent}. 
For a vertex $w \in V(\Gamma)$, the graph  $\Gamma - w$ is defined as the graph with vertex set $V(\Gamma) - \{w\}$ and edge set $E(\Gamma)- \{\{v,w\} \mid v \in V(\Gamma) \}$. 

A \emph{path} or an \emph{edge-path} from a vertex $v_0$ to a vertex $v_n$ of $\Gamma$ is a sequence of vertices $[v_0, v_1 \dots , v_n]$, where ${v_i}$ and $v_{i+1}$ are adjacent for all $i \in \{0, \dots , n-1 \}$. Its \emph{reverse-path} would be $[v_n, v_{n-1} \dots , v_0]$. A \emph{subpath} of the path $[v_0,v_1,\ldots,v_n]$ is a path of the form $[v_i,v_{i+1},\ldots ,v_j]$ for some $0\leq i < j\leq n$, or of the form $[v_i]$ for some $0\leq i\leq n$. The length of a path is one less than the total number of vertices in the sequence. If no vertex on a path appears in the sequence more than once, the path is called an \emph{embedded path}. The concatenation of two paths $\alpha = [u_0, u_1 \dots , u_n]$ and $\beta = [v_0, v_1 \dots , v_m] $ such that $u_n=v_0$ is $[\alpha,\beta] = [u_0, u_1 \dots , u_n, v_1 \dots , v_m]$. Analogously, the concatenation of a vertex $a$ and a path $\alpha = [u_0, u_1 \dots , u_n]$ such that $a$ and $u_0$ are adjacent is  $[a,\alpha]=[a,u_0, u_1 \dots , u_n]$.

A graph is \emph{connected} if there is a path between any two vertices. In a connected graph, the \emph{path-distance} between vertices   is the length of the shortest path between them; this defines a metric on the set of vertices called  the \emph{path metric}. For a graph $\Gamma$, we denote this distance by $\dist_\Gamma$. An \emph{$(L,C)$-quasi-isometry} $q\colon \Gamma\to \Delta$ between connected graphs is a function $q\colon V(\Gamma) \to V(\Delta)$ such that 
\[ \frac{1}{L}\dist_\Gamma(x,y) - C \leq \dist_\Delta(q(x),q(y)) \leq L\dist_\Gamma(x,y)+C\]
for any $x,y\in V(\Gamma)$, and for any $z\in V(\Delta)$ there is $x\in V(\Gamma)$ such that $\dist_\Delta(q(x),z)\leq C$. 

A graph with a $G$-action is called  \emph{a $G$-graph}.  For a vertex $v$ of $\Gamma$, the \emph{$G$-stabilizer} of $v$ is 
$G_v=\{g \in G \colon g.v = v\}$, and the $G$-orbit of $v$ is $G.v=\{g.v\colon g\in G\}$. Define analogously   $G$-stabilizers and   $G$-orbits of edges. For a $G$-graph $\Gamma$, we denote by $V_\infty(\Gamma)$  the set of vertices    with infinite $G$-stabilizer.

\section{Hyperbolically Embedded Subgroups and Coned-off Cayley Graphs}\label{sec:HypEmbDef}

\begin{definition}[Hyperbolically Embedded Subgroups]\cite{osin2}\label{def:hypemb}
Let $G$ be a group, let $H$ be a subgroup, and let $X\subset G$. Suppose that $G$ is generated by $X\cup H$, that is, $X$ is a \emph{relative generating set} of $G$ with respect to $H$. Let $\Gamma(G, X\sqcup   H)$ be the 
Cayley graph of $G$ whose edges are labeled by letters from the 
alphabet $X\sqcup H$. 
Note that $\Gamma(G, X\sqcup H)$ is not a simplicial graph. 
For  $h,k\in H$, let $\hat d_H (h,k)$ be the length of the shortest edge-path from $h$ to $k$ with the property that if an edge is labeled by an element of $H$ then the endpoints of the edge are not elements of the subgroup $H$ (this type of path is called admissible); if there is no  admissible path between $h$ and $k$, let $\hat d_\lambda (h,k)=\infty$. 
The   subgroup $H$ is \emph{hyperbolically embedded in $G$ with respect to $X$}, denoted as $H \hookrightarrow_h (G,X)$, if 
\begin{enumerate}
    \item $G$ is generated by $X\cup H$,
    \item $\Gamma(G,X\sqcup H)$ is a hyperbolic graph, and
    \item  the metric space $(H, \hat d_\lambda)$ has the property that any ball of finite radius has finitely many elements.
\end{enumerate}
If $H \hookrightarrow_h (G,X)$ for some $X$, then we write $H \hookrightarrow_h G$. 
\end{definition}

The conclusion of Proposition~\ref{q2} is contained in the following lemma.  

\begin{lem}
\label{prop:ConedOffCayleyGraph} \label{khat}  \label{lem3.2} \label{cqc} \label{ccg}
Let $G$ be a group, $X\subseteq G$, and $H\leq G$ an infite subgroup. Then 
\begin{enumerate}
    \item $G=\langle X \cup H \rangle$ if and only if $\hat{\Gamma}(G,H,X)$ is connected if and only if $\Gamma(G, H\cup X)$ is connected.
    \item If $G=\langle X\cup H \rangle$, then $\hat{\Gamma}(G,H,X)$ is quasi-isometric to ${\Gamma}(G,H \cup X)$.
    \item $\hat\Gamma(G,H,X)$ is fine at $G/H$ if and only if $(H,\hat d_H)$ is locally finite.
    \item $H\hookrightarrow_h (G,X)$ if and only if $\hat\Gamma(G,H,X)$ is   connected, hyperbolic and fine at cone vertices.
\end{enumerate}
\end{lem}

\begin{proof} Consider the inclusion of the vertex set of $\Gamma(G, H\cup X)$ into the vertex set of $\hat \Gamma(G,H,X)$. For any $a,b\in G$,
they are adjacent by an edge of 
$\Gamma(G, H\cup X)$ with label in $H$ if and only if $a$ and $b$ are both adjacent to the same cone vertex of $\hat\Gamma(G,H,X)$;  and they are adjacent by an edge with label in $X$ if and only if they are adjacent in $\hat\Gamma(G,H,X)$.  It follows that
$\hat\Gamma (G,H,X)$ and    $\Gamma(G,H\cup X)$ are both connected, or both disconnected. Hence, if both graphs are connected,
\[ \frac12\dist_{\hat\Gamma}(x,y)\leq \dist_{\hat\Gamma}(x,y) \leq 2\dist_\Gamma(x,y) \]
for any $x,y\in G$. In particular,  the inclusion $G\hookrightarrow G\cup G/H$ is a quasi-isometry $\hat\Gamma (G,H,X ) \to \Gamma(G,H\cup X)$.  For the third statement, observe that the metric spaces $(H, \hat d_H)$ and $(T_H \hat \Gamma, \angle_H)$   have the same underlying set and
\[ \frac12 \hat d_H (h,k) \leq   \angle_H(h,k)  \leq 2\hat d_H (h,k) \]
for any $h,k\in H$.
Moreover,  the metric spaces $(T_{gH}\hat\Gamma, \angle_{gH})$ for $gH\in G/H$ are all isometric. Therefore $(H, \hat d_H)$ is a locally finite metric space if and only if $\hat \Gamma$ is fine at $G/H$.
\end{proof}

\section{Edge-attachments to $(G,H)$-graphs} \label{sec:ModifyingGH-graphs}

In this section, we prove Proposition~\ref{lem2.2hs} and Corollary~\ref{attach} stated below.  

\begin{definition}
Let $\Gamma$ and $\Gamma'$ be $G$-graphs such that $V(\Gamma)=V(\Gamma')$ as $G$-sets. Let $u,v\in V(\Gamma)$ distinct vertices such that $\{u,v\} \not \in E(\Gamma)$. If
 \begin{align*}
V(\Gamma') &= V(\Gamma),\\
E(\Gamma') &= E(\Gamma) \cup \left\{\{g.u,g.v\} \mid g \in G \right\},
\end{align*}
then we say that $\Gamma'$ is \emph{obtained from $\Gamma$ by attaching a $G$-orbit of edges with representative $\{u,v\}$}. 
\end{definition}

\begin{proposition} \label{lem2.2hs}
Let $\Gamma$ be a connected  $G$-graph such that the $G$-stabilizers of edges are finite, and let $u,v \in V(\Gamma)$ such that $u \neq v$. Suppose that $\Gamma'$ is a $G$-graph obtained from $\Gamma$ by  attachment a  $G$-orbit of edges with   representative  $\{u,v\} \in V(\Gamma)$. Then
\begin{enumerate}
    \item If $a \in V(\Gamma)$ and $\Gamma$ is fine at $a$, then $\Gamma'$ is fine at $a$.
    \item The graph $\Gamma'$ is connected and the inclusion $\imath\colon\Gamma \hookrightarrow \Gamma'$ is a quasi-isometry.
\end{enumerate}

\end{proposition}

The first statement of  Proposition~\ref{lem2.2hs} is the main technical result of this note. It is a generalization of~\cite[Lemma 2.9]{MPW} and~\cite[Lemma 2.3]{bowditch} where the results are proved under the assumption that the graph $\Gamma$ is fine. The weaker assumption of having fineness at only a subset of vertices requires a different and non-trivial proof strategy. 
  
We record a corollary that is  used in the proof of the main result of the note.  

\begin{corollary}\label{attach}
Let $\Gamma$ be a $(G,H)$-graph and let $u,v \in V(\Gamma)$ such that $u \neq v$. If $\Gamma'$ is a $G$-graph obtained from $\Gamma$ by the  attachment of a  $G$-orbit of edges with  representative $\{u,v\}$, then $\Gamma'$ is a $(G,H)$-graph and the inclusion $\imath\colon\Gamma \hookrightarrow \Gamma'$ is a quasi-isometry.
\end{corollary}

The proof of Proposition~\ref{lem2.2hs} is discussed in subsection~\ref{subsec:proofAttaching}, and the proof of Corollary~\ref{attach} is in subsection~\ref{subsec:CorAttaching}

\subsection{A couple of Lemmas on Fine Graphs}

Let $\Gamma$ be a graph. A path $[u, u_1 \dots , u_n]$ in $\Gamma$ is an   \emph{escaping path from $u$ to $v$} if $v=u_i$ for some $i\in \{1, \dots , n\}$, and $u_i \neq  u$ for every $i \in \{1, \dots , n\}$.
For vertices $u$ and $v$ of $\Gamma$ and $k \in \mathbb{Z}_{+}$,  define:
\[
\vec{uv} (k)_\Gamma = \{w \in T_u \Gamma \mid w \text{ belongs to an escaping path from }u\text{ to }v\text{ of length } \leq k\}.
\]
 
\begin{lemma} \label{fatv}
A graph $\Gamma$ is fine at $u \in V(\Gamma)$ if and only if $\vec{uv}(k)_\Gamma$ is a finite set for every integer $k>0$ and every vertex $v \in V(\Gamma)$. 
\end{lemma}
\begin{proof}
Observe that for any $u,v\in V(\Gamma)$, $k>0$ and any $w\in \vec{uv}(k)_\Gamma$,
\[ \vec{uv}(k)_\Gamma \subseteq B_{T_u \Gamma}(w, 2k-2) \quad \text{ and } \quad    B_{T_u \Gamma}(w, k) \subseteq \vec{uw}(k+1),\]
where $B_{T_u \Gamma}(w, r)$ denotes the closed ball in $(T_u \Gamma,\angle_u)$ centered at $w$ of radius $r$. The statement of the lemma is then an immediate consequence.
\end{proof}

\begin{lemma} \label{lem2.1hs}
Let $\Gamma$ be a connected  $G$-graph with finite edge stabilizers. Suppose that $\Gamma$ is fine at $u \in V(\Gamma)$. Then for any vertex $v \in V(\Gamma)$, $G_u \cap G_v$ is finite or $u = v$.
\end{lemma}

\begin{proof} Suppose $u\neq v$ and let $H$ denote $G_u\cap G_v$. Since $\Gamma$ is connected there is a minimum length path  $[u_0,u_1,\ldots,u_k]$ from $u$ to $v$ with $k\geq1$. Let $U=\{gu_1 \colon g\in G_u\cap G_v\}$ and observe that this is an $H$-set. Since $H$ fixes both $u$ and $v$, it follows that 
$U \subseteq B_{T_u\Gamma}(u_1, 2k-2)$. Therefore, by the assumption that $\Gamma$ is fine at $u$,   $U$ is finite $H$-set. It follows that for any vertex $w \in U$, the edge $e=\{u,w\}$ is fixed by a finite index subgroup $H\cap G_e$ of $H$. Since edge $G$-stabilizers are finite,  $G_e\cap H$ is finite subgroup and therefore $H$ is finite. 
\end{proof}

\subsection{Proof of Proposition~\ref{lem2.2hs} } \label{subsec:proofAttaching}

\begin{proof}[Proof of Proposition~\ref{lem2.2hs}] By Lemma~\ref{fatv}, to prove that $\Gamma'$ is fine at $a$  is enough to show that for every integer $k\geq 1$ and for every  $b \in V(\Gamma')$ with $b \neq a$, the set  $\vec{ab}(k)_{\Gamma'}$ is  finite. Fix $b\in V(\Gamma')=V(\Gamma)$, $b\neq a$, and $k\geq1$.

Let $\alpha$ be a minimal length embedded path from $u$ to $v$ in $\Gamma$, such a path exists since $\Gamma$ is connected. Suppose that the length of  $\alpha$ is $\ell$,
\[\alpha=[u,u_1,\ldots, u_{\ell-1}, v] .\]
Observe that if $\ell=1$ then $\Gamma=\Gamma'$ and there is nothing to prove. Assume that $\ell>1$.  Let $\hat{\alpha}$ be the reverse path from $v$ to $u$. Let 
\[ n = k\ell\]
and let 
\[X_0 = \{ x\in \vec{ab}(k)_{\Gamma'} \mid x \not\in T_a\Gamma \}\]
To prove that $\vec{ab}(k)_{\Gamma'}$ is finite, we   define inductively a sequence of finite subsets $\vec{ab}(n)_\Gamma=W_n\subseteq W_{n-1} \subseteq \cdots \subseteq W_1 $, then we show that $X_0$ is a finite set, and conclude by proving that $\vec{ab}(k)_{\Gamma'} \subseteq W_1 \cup X_0$.

For the rest of the proof, we use the following terminology:  A subpath of length two of a path $P$ is called a  \emph{corner of $P$}.  

Define subsets $W_1,W_2,\dots , W_{n}$ and $Z_1,Z_2,\dots , Z_{n-1}$ of $T_a \Gamma$ as follows. Let 
\begin{align*}
W_n &= \vec{ab} (n) _\Gamma.
\end{align*}
Suppose $W_j$ has been defined, and let 
\begin{align*}
Z_{j-1} &= W_j \cup \{z \in T_a \Gamma \mid \exists w\in W_j\ \exists g\in G \  \exists c \text{ corner   of }  \alpha \text{ or } \hat{\alpha}\ \text{such that } g.c = [z,a,w] \}. \\  
W_{j-1} &= W_j \cup \{w \in T_a \Gamma \mid \exists z \in Z_{j-1} \text{ such that } \angle_{T_a \Gamma} (z,w) \leq n \}.
\end{align*}
It is immediate that  
\begin{align} \label{w}
W_j \subseteq Z_{j-1} \subseteq W_{j-1},
\end{align}
since if  $z \in Z_{j-1}$ then $\angle_{T_a \Gamma} (z,z) = 0$ and hence $z \in W_{j-1}$.

\begin{lem}\label{g1}
For $j \leq n$, $Z_{j-1}$ and $W_j$ are finite sets. In particular, $W_1$ is finite.
\end{lem}
\begin{proof}
\emph{If $W_j$ is finite, then $Z_{j-1}$ is finite.} This is a consequence of the assumption that  $G$-stabilizers of edges in $\Gamma$ are finite.  By contradiction, assume that $Z_{j-1}$ is infinite and  $W_j$ is finite. Since there are finitely many corners in $\alpha$ and $\hat\alpha$, the pigeon-hole argument shows that there is $w\in W_j$ and there is  corner $c$ of $\alpha$ or $\hat\alpha$ for which there are infinitely many distinct  $g_0,g_1, \ldots \in G$
and infinitely many distinct $z_0,z_1,z_2,\ldots \in Z_{j-1}$ such that $g_i.c=[z_i,a,w]$. Note that for $i\geq 1$, the element $g_i g_{0}^{-1}$ stabilizes the edge $e=\{a,w\} \in E(\Gamma)$ and hence $G_e$ is infinite, a contradiction.  

\emph{If $Z_j$ is finite then $W_j$ is finite.}  Indeed, since $\Gamma$ is fine at $a$, balls  in $(T_a\Gamma, \angle_a)$ are finite; therefore if $Z_j$ is finite then $\{w \in T_a \Gamma \mid \exists z \in Z_{n-1} \text{ such that } \angle_{T_a \Gamma} (z,w) \leq n \}$ is finite and hence $W_j$ is finite.

The lemma follows since   Lemma~\ref{fatv} implies that $\vec{ab}_\Gamma (n) = W_n$ is  finite.
\end{proof}

Let $\gamma'$ be a  path in $\Gamma'$. Let $\gamma$ be the path in $\Gamma$ obtained by replacing each subpath of length one of the form $[g.u,g.v]$ or $[g.v,g.u]$ for some $g \in G$ by the path $g. \alpha$ or $g. \hat{\alpha}$ (make a choice of $g$ if necessary), respectively. Since this construction is used once more in the proof, we refer to the path $\gamma$  as the \emph{$\alpha$-replacement of $\gamma'$}.

\begin{lem} \label{g}
Let $\gamma'$ be an escaping path in $\Gamma'$ from $a$ to $b$ of length at most $k$, and let $\gamma$ be its $\alpha$-replacement.  Then
\[ \gamma'\cap T_a\Gamma    \subseteq \gamma \cap T_a\Gamma \subseteq W_1,\]
where $\gamma' \cap T_a\Gamma$ is the set of vertices of $\gamma'$ that belong to $T_a
\Gamma$, and $\gamma\cap T_a\Gamma$ is defined analogously.
\end{lem}
\begin{proof}
By construction, the set of vertices of $\gamma'$ is a subset of the set of vertices of $\gamma$. Hence $\gamma'\cap T_a\Gamma $ is a subset of  $ \gamma \cap T_a\Gamma$.

Observe that $\gamma$ is of the form $[a,\gamma_1, a,\gamma_2, a, \dots ,a, \gamma_m]$, where each $\gamma_i$ does not contain the vertex $a$. Note that $m\leq n$ and that $\gamma$ is not   escaping when $m>1$. Let $w_i$ and $z_i$ denote the initial and terminal vertices of $\gamma_i$, respectively. 
The inclusions~\eqref{w} and the observation that  $\gamma \cap T_a\Gamma =\bigcup_{i=1}^m\gamma_i \cap T_a\Gamma$  show that is enough to prove that  $\gamma_i \cap T_a\Gamma$ is a subset of $W_i$ for $1\leq i\leq m$. This is done inductively using the following five claims.

\emph{Claim 1. $w_m \in W_m$.} Note that  $[a,\gamma_m]$ is an escaping path of length at most $n$ from $a$ to $b$ in $\Gamma$. Therefore $w_m \in \vec{ab}_\Gamma (n) = W_n$. Since $m\leq n$, it follows that  $w_m \in W_n \subseteq W_m$.

\emph{Claim 2. $z_{m-1} \in Z_{m-1}$.} Note that 
$[z_{m-1},a,w_m]$ is the translation of a corner of $\alpha$ or $\hat{\alpha}$;   since $w_m \in W_m$, by the definition, $z_{m-1} \in Z_{m-1}$.

\emph{Claim 3. If $z_{i+1} \in Z_{i+1}$  then $w_{i+1}\in W_{i+1}$}.
Indeed, since $\angle_{T_a \Gamma} (z_{i+1}, w_{i+1}) \leq n$ and $z_{i+1} \in Z_{i+1}$, it follows that $w_{i+1} \in W_{i+1}$.

\emph{Claim 4. If $w_{i+1}\in W_{i+1}$ then $z_i\in Z_i$.}  As $[z_{i},a,w_{i+1}]$ is the translation of a corner of $\alpha$ or $\hat{\alpha}$, if  $w_{i+1} \in W_{i+1}$, then by the definition $z_{i} \in Z_{i}$.

\emph{Claim 5. If $z_{i+1} \in Z_{i+1}$  then $\gamma_{i+1} \cap T_a\Gamma$ is a subset of $W_{i+1}$.}
Let $x\in \gamma_{i+1} \cap T_a\Gamma$. Observe that $\angle_{T_a \Gamma} (z_{i+1}, x) \leq n$. Since $z_{i+1} \in Z_{i+1}$, it follows that $x \in W_{i+1}$.

To conclude, observe that the first four claims imply that $z_i\in Z_i$ for $1\leq i\leq m$. Then the last claim imply that $\gamma_i\cap T_a\Gamma$ is a subset of $W_i \subset W_1$ for $1\leq i\leq m$.  
\end{proof}

\begin{lemma}\label{lem:eduardo}
$X_0$ is a finite set. 
\end{lemma}
\begin{proof}
Suppose that $X_0$ is an infinite set. 
Note that if $x\in X_0$ then there is $g\in G$ such that either $g.u=a$ and $g.v=x$, or $g.u=x$ and $g.v=a$. Without loss of generality, assume that  the set \[Y=\{g\in G \mid   \text{ $g.u=a$ and $g.v \in \vec{ab}(k)_{\Gamma'}$ and $g.v \not\in \vec{ab}(k)_{\Gamma}$} \}\] is infinite. 

For each $g\in Y$, let $\gamma_g'$ be an escaping path from $a$ to $b$ of length $\leq k$ in $\Gamma'$ that  contains   $[a,g.v]$ as an initial subpath. Let $\gamma_g$ be the $\alpha$-replacement of $\gamma_g'$. Since $\alpha=[u,u_1,\ldots, u_{\ell-1}, v]$, observe that $g.\alpha = [a,g.u_1,\ldots, g.u_{\ell-1}, g.v]$ is an initial subpath of $\gamma_g$.  

By Lemmas~\ref{g1} and~\ref{g},  $W_1$ is finite and $\gamma_g \cap T_a\Gamma \subseteq W_1$ for any $g\in Y$. It follows that  there exists a vertex $w \in W_1$ such that $[a,w]$ is an initial subpath of  $\gamma_g$ for infinitely  many distinct elements $g\in Y$. Therefore, there are infinitely many distinct  elements $g\in Y$ that map the edge $\{u,u_1\}$ to $\{a,w\}$, but that means that the edge $\{a,w\}$ of $\Gamma$ has infinite $G$-stabilizer. This is a  contradiction with the hypotheses on the $G$-graph $\Gamma$.  
\end{proof}

\begin{lemma}\label{lem:Farhan} 
$\vec{ab}(k)_{\Gamma'}\subseteq W_1\cup X_0$.
\end{lemma}
\begin{proof}
Let $\gamma'$ be an escaping path from $a$ to $b$ in $\Gamma'$ of length $\leq k$. Then Lemma~\ref{g} implies that $\gamma'\cap T_a\Gamma \subseteq W_1$. Since any vertex in $\gamma'\cap T_a\Gamma'$ is either in $T_a\Gamma$ or $X_0$, it follows that $\gamma'\cap T_a\Gamma' \subseteq W_1\cup X_0$ which proves the statement.
\end{proof}
By Lemmas~\ref{g1} and~\ref{lem:eduardo}, the set $W_1\cup X_0$ is   finite, hence Lemma~\ref{lem:Farhan} implies that $\vec{ab}(k)_{\Gamma'}$ is finite; this completes the proof that $\Gamma'$ is fine. It is left to address the second item of Proposition~\ref{lem2.2hs}.  

\begin{lemma}
$\Gamma'$ is connected and the inclusion $\Gamma \to \Gamma'$ is a quasi-isometry.
\end{lemma}
\begin{proof}
Since $V(\Gamma)=V(\Gamma')$ and $\Gamma$ is connected, the inclusion $\imath\colon \Gamma \hookrightarrow \Gamma'$ implies that $\Gamma'$ is also connected.
Let $x,y \in V(\Gamma)$. Let $\gamma'$ be a minimal length  path from $x$ to $y$ in $\Gamma'$, and let $\gamma$ be the $\alpha$-replacement of $\gamma'$. Then the length of $\gamma$ is at most $\ell \dist_{\Gamma'}(x,y)$. It follows that
\[ \dist_{\Gamma'}(x,y) \leq \dist_\Gamma (x,y) \leq \ell \dist_{\Gamma'}(x,y),\]
and hence $\imath\colon \Gamma \hookrightarrow \Gamma'$ is a quasi-isometry.
\end{proof}
This completes the proof of Proposition~\ref{lem2.2hs}. 
\end{proof}

\subsection{Proof of Corollary~\ref{attach}}\label{subsec:CorAttaching}

\begin{proof}[Proof of Corollary~\ref{attach}]
Definition~\ref{def:GHgraph} of $(G,H)$-graph has five items that need to be verified for $\Gamma'$. Items (1) and (5)   follow from Proposition~\ref{lem2.2hs}. Items (2) and (3) are immediate from the definition of $\Gamma'$. Item (4) follows from item (5) and Lemma~\ref{lem2.1hs}. 
\end{proof}

\section{Proof of the Main Result and Final Remarks.}\label{sec:05}

\begin{definition}\label{def:thick} 
A $(G,H)$-graph $\Gamma$ is \emph{thick} if it satisfies the following conditions: 
\begin{enumerate}
    \item \label{eq:minimal} It contains vertices $u_0$ and $v_0$ such that    \[ \{u_0,v_0\} \in E(\Gamma),\qquad G_{u_0}=1,\quad \text{ and }\quad  G_{v_0}=H.\]
    
    \item \label{eq:minimal02} There  is   a finite relative generating set $S$ of $G$ with respect to $H$, and a collection $\{u_0,\ldots , u_\ell\}$ of representatives of $G$-orbits of vertices of $\Gamma$ with finite $G$-stabilizers
    such that 
    \[ \{u_0,s.u_0\} \in E(\Gamma)\quad \text{for all $s\in S$, } \quad  \text{ and }\quad  \{u_0, u_j\} \in E(\Gamma) .\]
\end{enumerate}
\end{definition}

\begin{proposition}\label{prop:Existencethick}
Suppose $G$ is finitely generated relative to the infinite subgroup $H$. If there exists a $(G,H)$-graph $\Gamma$ then there is a thick $(G,H)$-graph with a $G$-equivariant quasi-isometry $\Gamma\to \Delta$.
\end{proposition}

 \begin{remark}\label{rem:trivialStabilizer}
 Let $\Gamma$ be a $(G,H)$-graph,  let  $\Gamma'$ be a $G$-graph, 
 and suppose $\Gamma$ is a $G$-subgraph of $\Gamma'$.
 If  there is an edge 
 $\{u,v\}\in E(\Gamma')$ such that $G_v$ is trivial, and
  \begin{align*}
V(\Gamma') &= V(\Gamma)\sqcup G.v,\\
E(\Gamma') &= E(\Gamma) \sqcup \left\{\{g.u,g.v\} \mid g \in G \right\},
\end{align*}
 then it is an observation that $\Gamma'$ is a $(G,H)$-graph and that the inclusion $\imath\colon\Gamma \hookrightarrow \Gamma'$ is a quasi-isometry. Hence, if there is a $(G,H)$-graph then there is a $(G,H)$-graph that  contains a vertex with trivial stabilizer.
\end{remark}

\begin{proof}[Proof of Proposition~\ref{prop:Existencethick}]
By Remark~\ref{rem:trivialStabilizer}, we can assume that $\Gamma$ contains a vertex $u_0$ with trivial stabilizer. Let $v_0\in V_\infty(\Gamma)$ such that $G_{v_0}=H$. Now we will 
perform a finite sequence of attachments $G$-orbits of edges to $\Gamma$. By Corollary~\ref{attach} the resulting graph $\Delta$ will be a $(G,H)$-graph with a $G$-equivariant quasi-isometry $\Gamma\to H$.
By attaching a  $G$-orbit of edges if necessary assume that $\{u_0,v_0\}\in E(\Gamma)$. Let $S$ be a finite relative generating set of $G$ with respect to $H$, and let $\{u_0,\ldots , u_\ell\}$ be a collection of representatives of $G$-orbits of vertices of $\Gamma$ with finite $G$-stabilizers. By adding finitely many orbits of edges if necessary, assume that $\{u_0,s.u_0\} \in E(\Gamma)$ for all $s\in S$ and $\{u_0, u_j\} \in E(\Gamma)$
for all $1\leq j\leq \ell$. Observe that the resulting graph is a thick $(G,H)$-graph.
\end{proof}

\begin{proposition}\label{prop:end}
Suppose that $H$ is an infinite subgroup of $G$. Let $\Gamma$ be a thick $(G,H)$-graph, and let $v_0$, $\{u_0,\ldots , u_\ell\}$ and $S$ be  as in Definition~\ref{def:thick}. Let 
\[X=\{g\in G \colon  \text{ $\dist_{\Gamma}(u_i, g.u_j)=1$ or $\dist_{\Gamma}(u_i, g.v_0)=1$ for some $0\leq i, j\leq\ell$ } \}\]
Then $q\colon \hat\Gamma(G,H,X) \to \Gamma$ given by  $g\mapsto g.u_0$ and $gH\mapsto g.v_0$ is a quasi-isometry. Moreover, $\hat\Gamma(G,H,X)$ is fine at cone vertices. In particular $\hat\Gamma(G,H,X)$ is a $(G,H)$-graph.
\end{proposition}

\begin{proof}
Note that if $f,g\in G$ and $\{f,g\}$ is an edge in $\hat\Gamma$ then
$\dist_\Gamma(q(f),q(g))\leq 3$. Indeed, if
$\{f.u_i, g.u_j\}$ is an edge in $\Gamma$ for some $i,j$, then 
$[f.u_0, f.u_i, g.u_j, g.u_0]$ is a path of length $3$ in $\Gamma$. Now observe that for any edge of the form  $\{g,gH\}$ of $\hat\Gamma$, we have that  $\{g.u_0,g.v_0\}$ is an edge of $\Gamma$ and hence $\dist_\Gamma(q(g),q(gH))=1\leq 3$. It follows that  
\[ \dist_{\Gamma}(q(a),q(b))\leq 3\dist_{\hat\Gamma}(a,b) \]
for any $a,b\in V(\hat\Gamma)$. 

On the other hand, any vertex of $\Gamma$ is of the form $g.u_j$ or $g.v_0$ for some $g\in G$ and $0\leq j\leq\ell$.  If $\{g_1.u_i,g_2.u_j\}$ is an edge in $\Gamma$, then $g_1^{-1}g_2\in X$ or equivalently $\{g_1, g_2\}$ is an edge of $\hat\Gamma$. Analogously, if $\{g_1.u_i,g_2.v_0\}$ is an edge in $\Gamma$, then $g_1^{-1}g_2\in X$ and hence $[g_1,g_2,g_2H]$ is path in $\hat\Gamma$. It follows that
\[ \dist_{\hat\Gamma}(a,b) \leq 2\dist_{\Gamma}(q(a),q(b))\]
for any $a,b\in V(\hat\Gamma)$. We have a $q\colon \hat\Gamma \to \Gamma$ is a $G$-equivariant quasi-isometry. It is left to prove that $\hat\Gamma$ is fine at cone vertices. 

The argument above shows that every edge in $\hat\Gamma$ of the form $\{g,gx\}$ with $g\in G$ and $x\in X$ corresponds to an (undirected) path of length at most three in $\Gamma$ between $g.u_0$ and $gx.u_0$ that does not contain any vertex in $V_\infty(\Gamma)$. Moreover, every edge in $\hat\Gamma$ of the form  $\{g,gH\}$ corresponds to the edge $\{g.u_0, g.v_0\}$ in $\Gamma$. In this way, to every path $\alpha$ in $\hat\Gamma$ from $e$ to $g$, corresponds a path $q(\alpha)$ from $u_0$ to $g.u_0$ with the property that $\alpha$ pass through the cone vertex $fH$ if and only if $q(\alpha)$ pass through the vertex $f.v_0$. Moreover $|q(\alpha)|\leq 3|\alpha|$. 

Suppose that $\hat\Gamma(G,H,X)$ is not fine at cone vertices. Then there is $r>0$ such that the set $\Omega=\{h\in H\colon \angle_H(e,h)\leq r\}$ is an infinite set. It follows that the set $\{h.u_0\in T_{v_0}\Gamma \colon h\in \Omega\}$ is an infinite set. For each $h\in \Omega$ there is path $\alpha_h$ from $e$ to $h$ that does not pass through the cone vertex $H$, and $|\alpha_h|\leq r$. Hence $\angle_{v_0}(u_0, h.u_0)\leq |q(\alpha_h)|\leq 3r$ for every $h\in \Omega$. Therefore $\Gamma$ is not fine. This is a contradiction, and therefore $\hat\Gamma$ is fine at cone vertices.
\end{proof}

\begin{remark}
A generalization of the argument above has been used in~\cite[Proof of Theorem C]{HMPS21} to show that, under some assumptions, fineness of coned-off Cayley graphs is preserved under quasi-isometries. In the same article, Proposition~\ref{q2} is used together with other results to show that if $H\hookrightarrow_h G$ and $G$ is finitely presented, then the relative Dehn function of $G$ with respect to $H$ is well defined, see~\cite[Proof of Theorem G]{HMPS21}.
\end{remark}

\subsection{A remark on hyberbolically embedded collections of subgroups} Let $G$ be a group and let $\mathcal{H}$ be a finite collection of subgroups. 

\begin{definition}
A graph $\Gamma$ is a $(G,\mathcal{H})$-graph if $G$ acts on $\Gamma$ and
\begin{enumerate}
\item $\Gamma$ is connected and hyperbolic,
\item there are finitely many $G$-orbits of vertices, 
\item $G$-stabilizers of vertices are finite or conjugates of subgroups in $\mathcal{H}$,
 and for every $H\in\mathcal{H}$ there is a vertex with $G$-stabilizer   equals $H$,
\item $G$-stabilizers  of edges are finite, and
\item  $\Gamma$ is fine at
$V_\infty (\Gamma)$.
\end{enumerate}
\end{definition}

\begin{definition}[Coned-off Cayley Graph]
For  $X\subset G$,  the \emph{Coned-off Cayley graph $\hat \Gamma(G,\mathcal{H},X)$} is the graph with vertex set the $G$-set $G\cup G/\mathcal{H}$, where
$G/\mathcal{H}$ is the set of left cosets $\{gH\colon g\in G\text{ and } H\in \mathcal{H}\}$ in $G$; and edge set $\{\{g,gx\}\colon g\in G \text{ and }  x\in X\} \cup \{ \{g, gH\} \colon g\in G \text{ and } H\in\mathcal{H} \}$.  
\end{definition}

In~\cite[Definition  2.9]{Osin16}, Osin defines the notion of  a collection of subgroups $\mathcal{H}$ being hyperbolically embedded into a group $G$.  The proofs of Proposition~\ref{q2} and Theorem~\ref{thm:main} work  mutatis mutandis to prove the following statements:

\begin{proposition} \label{q2-2}
Let $G$ be a group, $X\subset G$ and let $\mathcal{H}$ be a finite collection of infinite subgroups. Then $\mathcal{H} \hookrightarrow_h (G,X) $ if and only if   $\hat{\Gamma}(G,\mathcal{H},X)$ is connected, hyperbolic, and fine at cone vertices. 
\end{proposition}

\begin{thm}\label{thm:main2}
Let $G$ be a finitely generated group. A finite collection of infinite subgroups $\mathcal{H}$ is hyperbolically embedded in $G$  if and only if there exists a $(G,\mathcal{H})$-graph.
\end{thm}

\bibliographystyle{abbrv}
\bibliography{ref}

\end{document}